\documentclass[reqno]{amsart}
\usepackage{amssymb,amsmath,amsthm,amsfonts,fancyhdr}
\usepackage{indentfirst}
\usepackage{mathrsfs}
\usepackage{setspace}
\usepackage{color}
\usepackage{cite}
\usepackage{bm}
\usepackage{lineno}
\usepackage{latexsym}
\usepackage{verbatim}
\usepackage{cases}
\usepackage{esint}
\theoremstyle{plain} 
\newtheorem{thm}{Theorem}[section]

\newtheorem{lem}[thm]{Lemma}

\theoremstyle{definition}

\theoremstyle{remark}
\newtheorem{rmk}[thm]{Remark}

  \numberwithin{equation}{section}
  \numberwithin{figure}{section}

\begin{document}

\title{\textbf{ A note on the Liouville theorem of fully nonlinear elliptic equations }}

\author{Dongsheng Li}
\author{Lichun Liang}

\address{School of Mathematics and Statistics, Xi'an Jiaotong University, Xi'an, P.R.China 710049.}
\address{School of Mathematics and Statistics, Xi'an Jiaotong University, Xi'an, P.R.China 710049.}

\email{lidsh@mail.xjtu.edu.cn}
\email{lianglichun126@stu.xjtu.edu.cn}

\begin{abstract}
In this paper,  a new method is presented to investigate the asymptotic behavior of solutions to the fully nonlinear uniformly elliptic equation $F(D^2u)=0$ in exterior domains. This method does not depend on the $C^2$ regularity of $F$ and the dimension $n$.
\end{abstract}

\date{}
\maketitle

\section{Introduction}

\noindent

In this paper, we are concerned with the Liouville theorem of fully nonlinear elliptic equations of the form
\begin{equation}\label{eq1}
  F(D^2u)=0
\end{equation}
in the whole space and exterior domains. We assume that $F$ is uniformly elliptic, i.e., there are  two  constants $0<\lambda \leq \Lambda < \infty$ such that
$$\lambda \|N\|\leq F(M+N)-F(M)\leq \Lambda\|N\|$$
for any $M, N\in \mathcal{S}^{n\times n}$ with $N\geq 0$, where $\mathcal{S}^{n\times n}$ is the space of all real $n\times n$ symmetric matrices.

There have been a number of works on rigidity results for fully nonlinear elliptic equations.  In view of the Evans-Krylov estimate, it is easily seen that any entire quadratic growth solution $u$ of (\ref{eq1}) must be
a quadratic polynomial if $F$ is  convex or concave. In 2020, Li, Li and Yuan \cite{LLY20} made the following extension to the result: for $n\geq 3$ and $F\in C^2$, if the solution $u$ of (\ref{eq1}) in $\mathbb{R}^n \setminus \overline{B_1}$ has quadratic growth, then $u$ must be close to a quadratic polynomial at infinity with the error $|x|^{2-n}$. Furthermore, the error term $O(|x|^{2-n})$ was refined to $d|x|^{2-n}+O(|x|^{1-n})$ for some constant $d$ by Hong employing the Kelvin transformation in  \cite{Ho}.
Recently, for $n=2$, the asymptotic result was established with a logarithm term by Li and Liu \cite{LL24} using the quasiconformal mapping. These asymptotic results can be applied to some degenerate elliptic equations such as   Monge-Amp\`{e}re equations, special Lagrangian equations, quadratic Hessian equations and inverse harmonic Hessian equations.

In this paper, our purpose is to present a different proof for the asymptotic result. What's more important is that this proof  does not depend on the $C^2$ regularity assumption on $F$ and the dimension $n$.

We can now formulate our main result.
\bigskip

\begin{thm}\label{th1}
Let $F\in C^{1,\frac{2}{n}}(\mathcal{S}^{n\times n})$  be    uniformly elliptic and convex. Assume that $u$ is a viscosity solution of (\ref{eq1}) and satisfies
 \begin{equation}\label{eq5}
   |u(x)|\leq C|x|^2,\ \ \ x\in \mathbb{R}^n \setminus \overline{B_1}
 \end{equation}
 for some constant $C>0$.
 Then  there exist  some $A\in \mathcal{S}^{n\times n}$, $b, e \in \mathbb{R}^n$ and $c, d\in \mathbb{R}$ such that for any $0<\alpha<1$,

 $$u(x)=\frac{1}{2}x^{T}Ax+b\cdot x+d \Gamma (x)+c+\frac{e \cdot x}{\left|x^T(DF(A))^{-1}x\right|^{\frac{n}{2}}}+
 O(|x|^{1-n-\alpha})\ \ \ \mbox{as}\ \ \  |x|\rightarrow \infty,$$
 where
 \begin{equation*}
\Gamma (x)=\left\{
\begin{aligned}
 \left|x^T(DF(A))^{-1}x\right|^{\frac{2-n}{2}} &\ \ \ n\geq 3,\\
 \frac{1}{2}\log \left|x^T(DF(A))^{-1}x\right| & \ \ \ n=2.
\end{aligned}
\right.
\end{equation*}
 \end{thm}
\bigskip

\section{The proof of Theorem \ref{th1}}

\noindent

Now let us prove the following lemma.
\bigskip
\begin{lem}\label{le1}
Let $u$ be a positive solution of $a_{ij}D_{ij}u=0$ in $\mathbb{R}^2 \setminus \overline{B_1}$, where $a_{ij}$ is uniformly elliptic with the ellipticity constants $0<\lambda\leq \Lambda<\infty$ and $a_{ij}\rightarrow \delta_{ij}$  as $|x|\rightarrow \infty$. Then for any $0<\alpha<1$, there  exists a constant $R_\alpha>1$ depending only on $\alpha$ and $a_{ij}$ such that
$$0\leq u\leq  C|x|^{\alpha}, \ \ \ x\in \mathbb{R}^n \setminus \overline{B_{R_\alpha}}$$
for some constant $C>0$.
\end{lem}
\begin{proof}
\emph{Step 1:} We will show that $v:=|x|^{\alpha}$ is a subsolution of $a_{ij}D_{ij}u=0$ in $\mathbb{R}^2 \setminus \overline{B_{R_\alpha}}$ for some $R_{\alpha}>1$.

A simple computation yields
$$D_{ij}v=\alpha|x|^{\alpha-2}\delta_{ij}+\alpha(\alpha-2)|x|^{\alpha-4}x_ix_j$$
and $$\Delta v=\alpha^2|x|^{\alpha-2}.$$
Then we have
$$a_{ij}D_{ij}v=(a_{ij}-\delta_{ij})D_{ij}v+\Delta v\geq (\alpha^2-\epsilon(|x|)\alpha(3-\alpha))|x|^{\alpha-2},$$
where $\epsilon(|x|) \rightarrow 0$ as $|x|\rightarrow \infty$.
Hence letting $R_\alpha$ be chosen large enough such that
$$\alpha^2-\epsilon(|x|)\alpha(3-\alpha)\geq 0,\ \ \ |x|\geq R_\alpha,$$
we obtain the desired result.

\emph{Step 2:} We will prove that  $$a:=\sup_{x\in \mathbb{R}^2 \setminus \overline{B_{R_\alpha}}}\frac{u(x)}{v(x)}< +\infty.$$

To obtain a contradiction, suppose $a=+\infty$. Then for any constant $\varepsilon>0$, there exist  sequences of $\{R_i\}_{i=1}^{\infty}\subset \mathbb{R}^{+}$ increasing and $\{x_i\}_{i=1}^\infty$ with $R_1\geq R_\alpha$ and $x_i\in \partial B_{R_i}$ such that
$$\varepsilon u(x_i)\geq v(x_i).$$
Applying the Harnack inequality to $u$ in $\partial B_{R_i}$, we have
$$\varepsilon u(x)\geq C \varepsilon u(x_i)\geq Cv(x_i) =Cv(x),\ \ \ x\in \partial B_{R_i},$$
where $C$ depends only on $\lambda$, $\Lambda$ and $n$. Next, for any $x\in \mathbb{R}^2 \setminus \overline{B_{R_\alpha}}$ with $R_\alpha\leq |x|\leq R_i$ for some $i\geq4$, by the comparison principle in $B_{R_i} \setminus \overline{B_{R_\alpha}}$, we obtain
$$v(x)\leq \max_{ \partial B_{R_\alpha}}v+\frac{\varepsilon}{C}u(x), \ \ \ x\in \mathbb{R}^2 \setminus \overline{B_{R_\alpha}}.$$
Then letting $\varepsilon\rightarrow 0$ yields
$$|x|^\alpha=v(x)\leq \max_{ \partial B_{R_\alpha}}v=|R_\alpha|^\alpha, \ \ \ x\in \mathbb{R}^2 \setminus \overline{B_{R_\alpha}},$$
which leads to a contradiction.

\emph{Step 3:} We complete this proof.

By the definition of $a$, there exist some $R\geq R_\alpha$ and $x_0\in \partial B_{R}$ such that
$$u(x_0)\geq \frac{a}{2}v(x_0).$$
If $R_\alpha\leq R\leq R_\alpha+1$, then we have
$$a\leq \frac{2u(x_0)}{v(x_0)}=\frac{2u(x_0)}{|R|^\alpha}\leq 2 |R_\alpha|^{-\alpha}\|u\|_{L^{\infty}(B_{R_\alpha+1}\setminus B_{R_\alpha})}.$$
If $R\geq R_\alpha+1$, by the Harnack inequality in $\partial B_{R}$, then we have
$$u(x)\geq Cu(x_0)\geq C\frac{a}{2}v(x_0) =C\frac{a}{2}v(x),\ \ \ x\in \partial B_{R},$$
where $C$ depends only on $\lambda$, $\Lambda$ and $n$.
Furthermore, applying  the comparison principle yields
$$C\frac{a}{2} v(x)\leq C\frac{a}{2}|R_\alpha|^\alpha+u(x), \ \ \ x\in B_R \setminus \overline{B_{R_\alpha}}.$$
Taking $x\in  \partial B_{R_\alpha+1}$, we obtain
$$a\leq C u(x)\leq C \|u\|_{L^{\infty}(B_{R_\alpha+1}\setminus B_{R_\alpha})}.$$
\end{proof}
\bigskip

\begin{proof}[Proof of Theorem \ref{th1}]
By  the Evans-Krylov theorem, $u\in C^2$. We assume $F(0)=0$.
The proof will be divided into five steps.

\emph{Step 1:} We will show that there exists some $A\in \mathcal{S}^{n\times n}$ such that
$$D^2u(x)\rightarrow A \ \ \ \mbox{as}\ \ \  |x|\rightarrow \infty.$$

For any $e\in \mathbb{R}^n$ with $|e|=1$,  the pure second derivative $u_{ee}$ belongs to the subsolution set $\underline{S}(\lambda/n,\Lambda)$ (see \cite{CC95} for the definition). Therefore, by the aid of the weak Harnack inequality and the Evans-Krylov estimate, we can complete this step (see \cite[Lemma 2.1]{LLY20} for the detailed proof).

\emph{Step 2:} We will find some quadratic polynomial $Q(x)=\frac{1}{2}x^{T}\tilde{A}x+\tilde{b}\cdot x+\tilde{c}$ for some $\tilde{A}\in \mathcal{S}^{n\times n}$ with $F(\tilde{A})=0$, $\tilde{b}\in \mathbb{R}^n$ and $\tilde{c}\in \mathbb{R}$ such that $Q\leq u$ in $\mathbb{R}^n \setminus \overline{B_1}$.

For $k>1$, let $u_k$  be a viscosity solution of
\begin{equation*}
   \left\{
\begin{aligned}
   &F(D^2u_k)=0\ \ \mbox{in}\ \ B_k,\\
&u_k=u\ \  \mbox{on}\ \ \partial B_k.
\end{aligned}
\right.
 \end{equation*}
 Clearly, there exists infinite $k>1$ such that we have either
$$\max_{x\in \partial B_1}(u_{k}(x)-u(x))\geq 0$$
or
$$\min_{x\in \partial B_1}(u_{k}(x)-u(x))\leq 0.$$
Without loss of generality, we  assume  there exists infinite $k>1$ such that $M_k:=\max_{x\in \partial B_1}(u_{k}(x)-u(x))=u_{k}(x_k)-u(x_k)>0$ for $x_k\in \partial B_1$ and set

 \begin{equation*}
  \bar{u}_k(x)=u_{k}(x)-M_k
 \end{equation*}
 for all $x\in B_k$. Clearly, observing
 $$\bar{u}_k \leq u\ \ \mbox{on}\ \ \partial B_1\cup \partial B_{k},$$
 we apply the comparison principle to obtain
 \begin{equation}\label{eq6}
   \bar{u}_k \leq u \ \ \mbox{in}\ \ B_{k}\backslash \overline{B_1}.
 \end{equation}
For any $r>1$ and $k\geq 2 r$, applying the Evans-Krylov estimate and the Alexandroff-Bakelman-Pucci estimate to $F(D^2\bar{u}_k)=0$ in $B_k$, we have
$$\|D^2\bar{u}_k\|_{L^{\infty}(B_r)}\leq \|D^2u_k\|_{L^{\infty}(B_{k/2})}\leq \frac{C}{k^2}\|u_k\|_{L^{\infty}(B_{k})} \leq \frac{C}{k^2}\|u_k\|_{L^{\infty}(\partial B_{k})},
$$
where $C>0$ depends only on $n$, $\lambda$ and $\Lambda$. Combining this with the quadratic growth condition (\ref{eq5}), we know that $D^2\bar{u}_k$ is bounded in $B_r$ for $k\geq 2 r$. Now we consider the function
$$w_k(x)=\bar{u}_k(x)-D\bar{u}_k(x_k)\cdot (x-x_k)$$ and then $w_k(x_k)=|Dw_k(x_k)|=0$. Therefore, it is easily seen that
$$\|w_k\|_{L^{\infty}(B_r)}\leq C r^2$$
for some constant $C>0$ independent of $k$.
Furthermore, from the H\"{o}lder estimate applied to $F(D^2w_k)=0$ in $B_k$, it follows that there exists a subsequence of $\{w_k\}_{k\geq 2}^\infty$ that converges uniformly to a function $w\in C(\mathbb{R}^n)$ in compact sets  of $\mathbb{R}^n$.
Moreover, we have
$$F(D^2w)=0\ \ \mbox{in}\ \ \ \mathbb{R}^n$$
in the viscosity sense with
$$|w(x)|\leq C(1+|x|^2), \ \ x\in \mathbb{R}^n$$
for some constant $C>0$.  Therefore, by the   Evans-Krylov estimate, we see that
$w$ is a  quadratic polynomial $P$.

Recalling (\ref{eq6}), we have
\begin{equation}\label{eq7}
  w_k(x) \leq u(x)-D\bar{u}_k(x_k)\cdot (x-x_k), \ \ \ x\in  B_{k}\backslash \overline{B_1}.
\end{equation}
By taking some point $x_0\in \partial B_3(x_k)$ such that $x_0-x_k$ along the direction of $D\bar{u}_k(x_k)$, we obtain
$$3|D\bar{u}_k(x_k)|\leq u(x_0)-w_k(x_0)\leq C$$
 for some constant $C>0$ independent of $k$. Consequently, there exist some $b_\infty\in \mathbb{R}^n$ and $x_\infty \in \partial B_1$ such that, up to a subsequence,
 $$D\bar{u}_k(x_k) \rightarrow b_{\infty}\ \ \ \mbox{and}\ \ \ x_k \rightarrow x_\infty.$$
 Finally, letting $k\rightarrow \infty$ in (\ref{eq7}), we have
 $$P=w \leq u(x)-b_\infty\cdot (x-x_\infty), \ \ \ x\in  B_{k}\backslash \overline{B_1}.$$
 Hence $Q=P+b_\infty\cdot (x-x_\infty)$ is the desired quadratic polynomial.

\emph{Step 3:} We claim that $A=\tilde{A}$ and
\begin{equation}\label{eq4}
  |D^2u(x)-A|=O(|x|^{-n+\delta})\ \ \ \mbox{as}\ \ \  |x|\rightarrow \infty
\end{equation}
 for all $\delta>0$.

 Clearly, $u-Q$ satisfies the linear elliptic equation
 \begin{equation}\label{eq8}
   a_{ij}(x)D_{ij}(u-Q)(x)=0,\ \  \ x\in \mathbb{R}^n \setminus \overline{B_1},
 \end{equation}
where $a_{ij}(x)=\int_{0}^1 F_{ij}\left(tD^2u(x)+(1-t)D^2Q(x)\right)\,dt$. Since $F\in C^1$, we combine this with Step 1 to conclude that
$$a_{ij}(x) \rightarrow a_{ij}^{\infty}\ \ \ \mbox{as}\ \ \  |x|\rightarrow \infty,$$
where $a_{ij}^{\infty}=\int_{0}^1 F_{ij}\left(tA+(1-t)\tilde{A}\right)\,dt$.

\emph{Case $n\geq3$.} By the asymptotic theorem established by Li, Li and Yuan \cite[Theorem 2.2]{LLY20} for positive solutions of non-divergence linear elliptic equations, there exists a constant $c$ such that
\begin{equation}\label{eq2}
  u(x)=Q(x)+c+o(|x|^{2-n+\delta}) \ \ \ \mbox{as}\ \ \  |x|\rightarrow \infty
\end{equation}
 for all $\delta>0$.

 For any $x\in \mathbb{R}^n \setminus \overline{B_1}$ with $|x|>2$, let
 $$v(y)=\left(\frac{2}{|x|}\right)^2(u-Q-c)\left(x+\frac{|x|}{2}y\right),\ \  \ y\in B_1.$$
 Then we have $$F(D^2v(y)+\tilde{A})=0,\ \  \ y\in B_1.$$
 Applying the Evans-Krylov estimate, we obtain
 $$|D^2u(x)-\tilde{A}|=|D^2v(0)|\leq C |x|^{-n+\delta}, \ \ \ x\in \mathbb{R}^n \setminus \overline{B_1}$$
 for some constant $C>0$,  which easily yields $A=\tilde{A}$ and (\ref{eq4}).

 \emph{Case $n=2$.} By Lemma \ref{le1}, we have
 \begin{equation}\label{eq3}
   u(x)=Q(x)+O(|x|^{\delta}) \ \ \ \mbox{as}\ \ \  |x|\rightarrow \infty
 \end{equation}
 for all $\delta>0$. Following
the same arguments to the case $n\geq 3$, we also obtain $A=\tilde{A}$ and (\ref{eq4}).

 \emph{Step 4:} The aim of this step is to refine the error to  $O(|x|^{2-n})$ $(n\geq 3)$ and $O(|x|^{-1})$ $(n=2)$
for  the asymptotic results (\ref{eq2}) and (\ref{eq3}).

Without loss of generality, we can assume $a_{ij}^{\infty}=\delta_{ij}$.
 We take $0<\delta<\frac{1}{2}$ small enough. Since $F\in C^{1,\frac{2}{n}}$, we have

 $$|a_{ij}(x)-\delta_{ij}|\leq C|x|^{\frac{2}{n}(-n+\delta)},\ \ \ x\in \mathbb{R}^n \setminus \overline{B_1}$$
 for some constant $C>0$.

 \emph{Case $n\geq3$.}
 Once again using the asymptotic theorem (\cite[Theorem 2.2]{LLY20}), we obtain
 $$u(x)=Q(x)+c+O(|x|^{2-n}) \ \ \ \mbox{as}\ \ \  |x|\rightarrow \infty.$$

  \emph{Case $n=2$.} Clearly, we have
  $$\Delta (u-Q)(x)=(\delta_{ij}-a_{ij}(x))D_{ij}(u-Q)(x)=O(|x|^{-2+\delta}|x|^{-2+\delta}) \ \ \ \mbox{as}\ \ \  |x|\rightarrow \infty.$$
  By \cite[Lemma 3.4]{LL24}, there exists some function
  $$w(x)=O(|x|^{-2+2\delta+\epsilon})$$
  for any $0<\epsilon<1$ such that $$\Delta (u-Q-w)(x)=0$$
  and $$(u-Q-w)(x)=O(|x|^\delta)\ \ \ \mbox{as}\ \ \  |x|\rightarrow \infty.$$
  Then applying \cite[Lemma 3.5]{LL24} leads to
  $$(u-Q-w)(x)=d \log|x|+c+O(|x|^{-1})$$
  for some $c, d\in \mathbb{R}$. Hence
  $$u(x)=Q(x)+d \log|x|+c+O(|x|^{-1})$$

  \emph{Step 5:} We will determine the term $\frac{x}{|x|^n}$.

  \emph{Case $n\geq3$.} We consider the function
  $$E(x)=u(x)-Q(x)-c=O(|x|^{2-n}) \ \ \ \mbox{as}\ \ \  |x|\rightarrow \infty.$$
  By the  Evans-Krylov estimate, we have
  $$ D^2E(x)=O(|x|^{-n}) \ \ \ \mbox{as}\ \ \  |x|\rightarrow \infty.$$
  Since $E$ satisfies the linear equation (\ref{eq8}) and $F\in C^{1,\frac{2}{n}}$, we obtain
  $$\Delta E(x)=(\delta_{ij}-a_{ij}(x))D_{ij}E(x):=g(x)=O(|x|^{-2}|x|^{-n}), \ \ \ x\in \mathbb{R}^n \setminus \overline{B_R}$$
  for some constant $R>1$.
  Hence the Kelvin transformation of $E$, defined by
$$K[E](x)=|x|^{2-n}E\left(\frac{x}{|x|^2}\right) \ \ \mbox{for}\ \ x\in B_{\frac{1}{R}}\setminus \{0\},$$
satisfies
$$\Delta K[E](x)=|x|^{-n-2}g\left(\frac{x}{|x|^2}\right):=\tilde{g}(x)
  \ \ \mbox{in}\ \ B_{\frac{1}{R}}\setminus \{0\},$$
where $\tilde{g}$ satisfies
$$|\tilde{g}(x)| \leq C \ \ \mbox{in}\ \ B_{\frac{1}{R}} \backslash \{0\}$$
for some constant $C>0$.

 Clearly, $\tilde{g}\in L^{\infty}(B_{1/R})$. By the $L^p$ regularity, we have $K[E] \in W^{2,p}(B_{1/R})$ for any $1<p<\infty$. Hence $K[E] \in C^{1,\alpha}\left(B_{1/R}\right)$ for any $0<\alpha \leq 1-\frac{n}{p}$ with $p\in (n,\infty)$, which follows from the Sobolev imbedding theorem. More precisely,  there exist some constant $\widetilde{C}>0$  and  linear function $\tilde{b}\cdot x+\tilde{c}$ such that
$$\left|K[E](x)-\tilde{b}\cdot x-\tilde{c}\right| \leq \widetilde{C}|x|^{1+\alpha}\ \ \mbox{in}\ \ B_{\frac{1}{R}}.$$
Therefore, going back to $E$, we have
$$\left|E(x)-\tilde{c}|x|^{2-n}- \frac{\tilde{b}\cdot x}{|x|^n}\right| \leq  \frac{\widetilde{C}}{|x|^{n-1+\alpha}}\ \ \mbox{in}\ \ \mathbb{R}^n\backslash \overline{B_{R}}.$$

\emph{Case $n=2$.} We let
$$E(x)=u(x)-Q(x)-d \log|x|-c=O(|x|^{-1}) \ \ \ \mbox{as}\ \ \  |x|\rightarrow \infty.$$
 The rest of the proof runs as before (or see  \cite[pp. 10-11]{LL24} for the detailed proof).
\end{proof}
\bigskip
\begin{rmk}
For $n\geq 3$, if $F\in C^1$, we can obtain the error term $o(|x|^{2-n+\delta})$ for any $\delta>0$; if $F\in C^{1,\gamma}$ for some $0<\gamma<1$, the error term is $O(|x|^{2-n})$.
When $F$ is Lipschitz continuous, Lian and Zhang \cite{LZP} obtained an asymptotic result with the error $O\left(|x|^{1-(n-1)\frac{\lambda}{\Lambda}}\right)$ that
is the fundamental solution of the Pucci's operator (see \cite{L01,ASS11} for more details).
 \end{rmk}

\bigskip
\section*{Acknowledgement}
This work is supported by NSFC 12071365.

\end{document}